\documentclass[10pt]{amsart}
\usepackage{latexsym,amssymb, amsthm}
\usepackage[mathscr]{euscript}
\usepackage{cite}
\usepackage[all]{xy}

\newcommand{\cb}{\overline{C}} 
\newcommand{\ccc}{\mathscr{C}}
\newcommand{\cliff}{{\rm Cliff}} 
\newcommand{\cp}{\ccc_{P}} 
\newcommand{\fff}{\mathscr{F}} 
\newcommand{\gon}{{\rm gon}} 
\newcommand{\mmp}{\mathfrak{m}_{P}} 
\newcommand{\ob}{\overline{\mathcal{O}}} 
\newcommand{\obp}{\overline{\mathcal{O}}_P}
\newcommand{\oo}{\mathcal{O}} 
\newcommand{\op}{\mathcal{O}_P} 
\newcommand{\pb}{\overline{P}} 
\newcommand{\sys}{\mathscr{L}}
\newcommand{\vl}{\mathscr{W}}
\newcommand{\vlp}{\mathscr{W}_{P}}
\newcommand{\ww}{\omega} 

\newtheorem{thm}{Theorem}[section]
\newtheorem{prop}{Proposition}[section]
\newtheorem{lem}[thm]{Lemma}
\newtheorem{exam}[thm]{Example}
\newtheorem{defi}[thm]{Definition}

\begin{document}

\title{Max Noether's theorem for integral curves}

\thanks{Part of this work corresponds to the first named author's Ph. D.
thesis \cite{A}. The third named author thanks Steven L. Kleiman for a stay
at MIT a couple of years ago, when the problem of this work arised, and for all learnt about it from him.
The first named author was supported by CNPq grant number 140315/2010-1.
The second named author is partially supported by CNPq grant number 486468/2013-5  and by FAPEAL.
The third named author is partially supported by CNPq grant number 307978/2012-5}



\author[Lia F. F. Abrantes]{Lia Feital Fusaro Abrantes}
\address{ Departamento de Matem\'atica, CCE, UFV
Av. P H Rolfs s/n, 36579-900, Vi\c{c}osa MG Brazil}
\email{liafeital@ufv.br}

\author[A. Contiero]{Andr\'e Contiero}
\address{Instituto de Matem\'atica, UFAL. Av. Lourval de Melo Mota, s/n, 57072-900 Macei\'o, Brazil}
\email{andrecontiero@mat.ufal.br}


 \author[R. V. Martins]{Renato Vidal Martins}
\address{Departamento de Matem\'atica, ICEx, UFMG
Av. Ant\^onio Carlos 6627,
30123-970 Belo Horizonte MG, Brazil.}
\email{renato@mat.ufmg.br}





\begin{abstract}
We generalize, for integral curves, a celebrated result of Max Noe\-ther on global sections of the
$n$-dualizing sheaf of a smooth nonhyperelliptic curve. This is our main result.
We also obtain an embedding of a non-Gorenstein curve in a way that we can express the dimensions
of the components of the ideal in terms of the main invariants of the curve. Afterwards we focus on
gonality, Clifford index and Koszul cohomology of non-Gorenstein curves by allowing torsion free
sheaves of rank 1 in their definitions. We find an upper bound for the gonality, which agrees with
Brill-Noether's one for a rational and unibranch curve. We characterize curves of genus 5 with
Clifford index 1, and, finally, we study Green's Conjecture for a certain class of curves, called
nearly Gorenstein.
\end{abstract}

\maketitle

\section{Introduction}\label{Intro}

In 1880, Max Noether established in \cite{N} a remarkable result which, in modern language, can be stated as follows.
\medskip

\noindent \textit{ {\rm {\bf Max Noether's Theorem}} \
If $C$ is a smooth, nonhyperelliptic curve which is complete over an algebraic closed field, and $\ww$ its dualizing sheaf, then the maps
\begin{equation}\label{TMN}
{\rm Sym}^n  H^0(C,\ww)\longrightarrow H^0(C,\ww^n)
\end{equation}are surjective for all $n\geq 1$.
}

\

One of its first applications, just taking $n=2$, is that a smooth canonical curve is set-theoretically intersection of $(g-2)(g-3)/2$
linearly independent quadrics, where $g$ is the genus of $C$.
In the late 1910's Enriques \cite{En}, proved that a canonical nonhypereliptic curve is the set-theorically
intersection of quadrics, unless the smooth curve is trigonal or isomorphic to a plane
quintic (a result also proved by Babbage in \cite{B}). A complete description of the canonical ideal
in terms of equations, based on Noether's dimension counts and following Enriques' division into cases,
was done by Petri \cite{P} in the early 1920's, as presented in \cite[p. 131]{ACGH}. Later on, new
approaches to Petri's analysis were carried out by Arbarello--Sernesi \cite{AS}, Mumford
\cite{M}, Saint--Donat \cite{S}, Shokurov \cite{Sh} and Stoehr--Viana \cite{SV}.

A proof of Max Noether's Theorem can also be found in \cite[p. 117]{ACGH},
where one can note that it is a consequence of projective normality of extremal (Castelnuovo) curves.
Indeed, extremal curves are projectively normal, which is a general fact proved in
\cite[pp.\,113-117]{ACGH} for smooth curves. But since Riemann--Roch
and Clifford's Theorems have versions for singular curves (\cite[App.]{EKS}, \cite[pp. 186-191]{R},
\cite[p. 108]{S}), the same proof holds for all integral curves as well. So if we assume that $C$ is Gorenstein,
i.e.,  its dualizing sheaf $\ww$ is invertible, then $\ww$ defines a morphism $\kappa :C\rightarrow{\mathbb{P}}^{g-1}$. Let $C':=\kappa (C)$
be the canonical model of $C$. Based on his Ph. D. thesis under Zariski, Rosenlicht proved in \cite{R} that $C'$ is extremal and that $\kappa$
is an isomorphism if $C$ is nonhyperelliptic. Therefore Max Noether's Theorem holds actually for all
Gorenstein nonhyperelliptic curves. An application of Max Noether's Theorem for Gorenstein curves, and also
Petri's Analysis, was pointed out by Mumford in \cite{M} to construct certain moduli spaces of curves with prescribed
Weierstrass semigroup, which was done by Stoehr in \cite{St1}, reproved and explored by him and the second
named author in \cite{CS}.

Although Max Noether's statement is purely intrinsic,  its proof for smooth, and more generally Gorenstein, curves is not.
As said above, the result is a straightforward consequence of the fact that canonical curves are extremal, a
strongly extrinsic argument. On the other hand, if the concern is proving for general integral curves,
it is likely difficult to avoid some hard local algebra. A step forward was done by Rosenlicht in his main theorem
\cite[Cor.\ and Thm.\,17, p.\,189]{R} in the late 1950's. In order to state it, one needs to extend a notion just
introduced above, i.e., if $\cb$ is the normalization of $C$ and $\overline{\kappa}:\cb\to\mathbb{P}^{g-1}$ the
morphism induced by the dualizing sheaf of $C$, one calls $C':=\overline{\kappa}(\cb)$ the
\emph{canonical model} of $C$.

\medskip

\noindent \textit{ {\bf Rosenlicht's Theorem}\ Let $C$ be an integral, nonhyperelliptic curve, which is complete
over an algebraic closed field. Then there exists a birrational morphism
\begin{equation}\label{ROT}
C'\longrightarrow C
\end{equation}
which is an isomorphism if and only if $C$ is Gorenstein.}

\

For the sake of simplicity, we will refer to the surjectivity of the morphisms in (\ref{TMN})
as ``Max Noether's statement", and to the existence of a birational morphism like (\ref{ROT})
as ``Rosenlicht's statement", no matter the hypotheses on the curve are.

According to \cite[Int.]{KM}, the Gorenstein part of the result, i.e., the isomorphism $C\cong C'$, was,
later on, successively reproved by several authors in many different ways: Deligne--Mumford \cite{DM} in 1969,
Mumford--Saint-Donat \cite{MSD} in 1973, Sakai \cite{Sak} in 1977, Catanese \cite[p.\,51]{C} in 1982, Fujita
\cite[p.\,39 Thm.\,(A1)]{F} in 1983, and Hartshorne \cite[Thm.\,1.6, p.\,379]{H} in 1986. More recently, Rosenlicht's Theorem
was also reproved within a modern language and refined version by Kleiman, with the third
named author, in \cite{KM}: if $\widehat{C}$ is the blowup of $C$ along its dualizing sheaf, then $\widehat{C}\cong C'$.

A connection between Max Noether's and Rosenlicht's statements in the general integral case appears in
\cite[Rem. 2.8]{Mt}, where it is proved that the former implies the latter. Therefore, since Noether's statement
is stronger, the technique of computing values of differentials, which is the core of Rosenlicht's proof,
becomes even harder if applied to prove Max Noether's Theorem. An attempt of doing
so was made in \cite[Thm. 3.7]{M}, but it assumes all non-Gorenstein points are unibranch, which simplifies
the combinatorial part of the proof. Removing this hypothesis is exactly what we do here in our main result:

\medskip

\noindent\textit{ {\rm {\bf Theorem 1}}\ Let $C$ be an integral, nonhyperelliptic curve, which is complete
over an algebraically closed field, and $\ww$ its dualizing sheaf.
Then the homomorphisms
$$
{\rm{Sym}}^n  H^0(C,\ww )\longrightarrow H^0(C,\ww ^n)
$$
are surjective for $n\geq 1$.
}

\

The proof is rather technical. The significant effort took to pass from the unibranch to the multibranch case can
be seen, for instance, in the statement of Lemma \ref{maroto}, which holds for unibranch non-Gorenstein points
by its very definition.

Max Noether's Theorem also corresponds positively to the case $p=0$ of Green's famous conjecture on
canonical curves. 
This naturally led us to the study of Koszul cohomology and  Clifford index, allowing torsion free sheaves of
rank 1 in their definitions. As a consequence, we have that the following five conditions are equivalent: (i)
$C$ is nonhyperelliptic; (ii) Rosenlicht's statement holds; (iii) Max Noether's statement holds;
(iv) $K_{0,2}(C,\ww)=0$; (v) ${\rm Cliff}(C)>0$ or $C$ is rational nearly normal (see Definition \ref{defnng}).
This is our Theorem \ref{ap1equiv}.

As mentioned above, once Max Noether's Theorem is introduced, it is natural to compute the dimensions
of the homogeneous components of the ideal of the canonical curve. It can be trivially read off Noether's
result. Our version for this in the case of certain non-Gorenstein curves is the following result,
obtained by blowdown procedures, as can be checked from its proof.

\medskip

\noindent \textit{ {\bf Theorem 2 }\ Let $C$ be an integral curve of genus $g$ which is complete over an algebraically closed field. Assume
the non-Gorenstein points are  unibranch at most. Let $\pi:\cb\to C$ be the normalization map, and
$\widetilde{\pi}:\widetilde{C}\to C$ be the partial dessingularization of the non-Gorenstein points.
 Let $\oo:=\oo_C$ be the structure sheaf, set $\ob:=\pi_{*}(\oo_{\cb})$ and
 $\widetilde{\oo}=\widetilde{\pi}_{*}(\oo_{\widetilde{C}})$. Then there exists an embedding
 $C\hookrightarrow \mathbb{P}^{g+2(\rho-\sigma)-1}$ such that
$$\dim(I_r(C))=\left(\begin{array}{c} r+g+2(\rho-\sigma)-1 \\
                                                    r \\ \end{array} \right)
  +g(1-2r)-2r(\rho-\sigma)+r-1 $$
  where $\rho=h^0(\ob/\mathcal{H}{\rm om}(\ob,\oo))-h^0(\widetilde{\oo}/\mathcal{H}{\rm om}(\widetilde{\oo},\oo))$
   and $\sigma=h^0(\ob/\oo)-h^0(\ob/\widetilde{\oo})$.
In particular,
$$\dim(I_2(C))= \displaystyle\frac{(g+2(\rho-\sigma)-1)(g+2(\rho-\sigma)-2)-2g}{2}. $$}

As we did for Clifford index and Koszul cohomology, we also allow torsion free sheaves of rank 1 in the
definition of gonality. From a geometric perspective, this corresponds to replace morphisms by pencils.
Our results concerning these three concepts are summarized in the following statement.

\medskip

\noindent \textit{ {\rm \bf {Theorem 3}} \ Let $C$ be a non-Gorenstein integral curve of genus $g$
which is complete over an algebraically
closed field.}
\begin{itemize}
\item[(i)] It holds $$2\leq {\gon}(C)\leq g$$
and if the upper bound is attained, then $C$ is Kunz with only one non-Gorenstein point, and either $C$ is
rational or $\cb$ is elliptic.
\item[(ii)] If $C$ is rational with a unique non-Gorenstein point, which is unibranch, then
$$
\gon(C)\leq \big\lfloor\frac{g+3}{2}\big\rfloor.
$$
\item[(iii)] for $g=5$, ${\rm Cliff}(C)=1$ if and only if $C$ is trigonal or there exists $\fff$ on $C$ such that
$\deg(\fff)=5$ and  $h^0(\fff)=3$; moreover, the latter condition is necessary;
\item[(iv)] If $C$ is nearly Gorenstein, then $K_{p,2}(C,\ww)=0$ for every $p<\eta$. Moreover, there
exists a family of curves $\{C_p\}_{p\geq 1}$ such that $\cliff(C_p)=1$ and $K_{p,2}(C_p,\ww)=0$.
\end{itemize}

The terms used above are in Definition \ref{defnng}. As the equivalence (iv)$\Leftrightarrow$(v) of
Theorem \ref{ap1equiv} shows that Green's assertion fails to hold if $p=0$, the family constructed in
the item (vii) above shows that it fails to hold for arbitrary $p\geq 1$ as well.

\




\section{Preliminaries}\label{prelim}

Let $C$ be a complete integral curve of arithmetic genus $g$ defined over an algebraically closed field with structure sheaf $\oo_C$,
or simply $\oo$. A \emph{linear system of dimension $r$ on $C$} is a set of the form
$$
\sys =\sys(\fff ,V):=\{x^{-1}\fff\ |\ x\in V\setminus 0\}
$$
where $\fff$ is a coherent fractional ideal sheaf on $C$ and $V$ is a vector subspace of
$H^{0}(\fff )$ of dimension $r+1$.

The notion of linear systems on curves presented here
is characterized by interchanging bundles by torsion free sheaves of rank $1$. This is a convenient
 approach for singular curves since they can possibly admit \emph{non-removable} base points, see for instance M.
Coppens' \cite{Cp}.

 The \emph{degree} of the linear system $\sys$ is the integer
$d:=\deg \fff :=\chi (\fff )-\chi (\oo)$, where $\chi$ denotes the Euler characteristic. Note, in particular, that if $\oo\subset\fff$ then
$$\deg\fff=\sum_{P\in C}\dim(\fff_P/\op).$$
The notation $g_{d}^{r}$ stands for a linear system of degree $d$ and dimension $r$.
The linear system is said to be \emph{complete} if $V=H^0(\fff)$, in this case one simply writes $\sys=|\fff|$.

Recall that a point $P\in C$ is \emph{Gorenstein} if the stalk $\ww_P$ is a free $\oo_P$-module,
where $\ww$ stands for the dualizing sheaf on $C$. 
The curve $C$ is said to be \emph{Gorenstein} if all of its points are so, or equivalently,
$\ww$ is invertible.



According to E. Ballico's \cite[p. 363, Dfn. 2.1 (3)]{Bal}, the gonality of C is the smallest d for which there exists a
$g_{d}^{1}$ in C, or equivalently, a torsion free sheaf $\fff$ of rank $1$ on $C$ with degree $d$ and
$h^0(\fff)=2$. A geometric motivation for this definition can be found, for instance, in \cite{RSt} for Gorenstein curves
 and \cite[Ex. 2.2]{AM} for non-Gorenstein ones.



Given a sheaf $\mathscr{G}$ on $C$, if $\varphi :\mathcal{X}\to C$ is a morphism from a scheme $\mathcal{X}$ to $C$, then we set
$\oo_{\mathcal{X}}\mathscr{G}:=\varphi^* \mathscr{G}/\rm{Torsion}(\varphi^*\mathscr{G})$. 
For each coherent sheaf $\fff$ on $C$ we set $\fff^n:=\rm Sym^n\fff/\rm Torsion (\rm {Sym}^n\fff)$. In particular,
if $\fff$ is invertible then clearly $\fff^{n}=\fff^{\otimes n}$.

Let us consider the normalization map $\pi :\cb\rightarrow C$.
In \cite[p.\,188\,top]{R} Rosenlicht showed that the linear
system $\sys(\oo_{\cb}\ww,H^0(\ww))$
is base point free. 
He considered then the induced morphism $\psi :\cb\rightarrow{\mathbb{P}}^{g-1}$
and called its image $C':=\psi(\cb)$ the canonical model of $C$.  Additionally, Rosenlicht also proved \cite[Thm.\,17]{R}
that if $C$ is nonhyperelliptic, then the map $\pi :\cb\rightarrow C$
factors through a map $\pi' : C'\rightarrow C$. So set $\oo':=\pi'_*(\oo_{C'})$ in this case.

Let $\widehat{C}:=\rm {Proj}(\oplus\,\ww ^n)$ be the blowup of $C$ along $\ww$ and
$\widehat{\pi} :\widehat{C}\rightarrow C$ be the natural morphism.
Set $\widehat{\oo}=\widehat{\pi}_*(\oo_{\widehat{C}})$ and $\widehat{\oo}\ww:=\widehat{\pi}_*(\oo _{\widehat{C}}\ww)$.
In \cite[Dfn.\,4.9]{KM} one finds
another characterization of the canonical model $C'$, it is the image of the morphism
$\widehat{\psi}:\widehat{C}\rightarrow{\mathbb{P}}^{g-1}$ defined by the linear system
$\widehat{\sys}(\oo_{\widehat{C}}\ww,H^0(\ww))$. By Rosenlicht's Theorem, since $\ww$ is
generated by global sections, we have that $\widehat{\psi}:\widehat{C}\rightarrow C'$ is an
isomorphism if $C$ is nonhyperelliptic.

It is known that the sheaf $\overline{\oo}\ww:=\pi_*(\oo_{\cb}\ww)$ can be generated by
the global sections of $\ww$, see \cite[p.\,188 top]{R}.
Since there are only a finite number of singular
points on $C$ and the ground field is infinite, one concludes that there is a
differential $\zeta\in H^0(\ww)$ such that
$(\overline{\oo}\ww)_P=\zeta\cdot\overline{\oo}_P$ for every singular point $P\in C$, where
$\overline{\oo}:=\pi _{*}(\oo _{\cb})$. This leads us to introduce the sheaf
$$\vl=\vl_{\zeta}:=\ww/\zeta \, .$$

\noindent If we consider $\ccc:=\mathcal{H}\rm {om}(\overline{\oo},\oo)$
the conductor of $\overline{\oo}$ into $\oo$, then we see that
$$
\ccc_P\subset
\oo_P \subset \vlp\subset\widehat{\oo}_P=\op'\subset\obp
$$
for every singular point $P\in C$, where the equality makes sense if and only if $C$ is nonhyperelliptic.
Set also $H^0(\vl)^n$ to  be the set of all finite sums of products of $n$ elements
from $H^0(\vl)$, and we just warn the reader not to confuse it with $H^0(\vl)^{\oplus n}$
which is different.

\begin{defi} \label{defnng}
\rm{Let $P\in C$ be any point. Set
$$
\eta_P:=\dim(\vlp/\op)=1\ \ \ \ \ \ \ \ \ \ \ \mu_P:=\dim({\oo_{P}'}/\vlp)
$$
and also
$$
\eta:=\sum_{P\in C}\eta_P\ \ \ \ \ \ \ \ \ \ \mu:=\sum_{P\in C}\mu_P
$$
Following \cite[pp. 418, 433, Prps. 21, 28]{BF} call $P$ \emph{Kunz} if $\eta_P=1$ and,
accordingly, say $C$ is \emph{Kunz} if all of its non-Gorenstein points are so; quite similarly,
say a non-Gorenstein point is \emph{almost Gorenstein} if $\mu_P=1$. Any Kunz point is almost
Gorenstein as well by \cite[Prp. 21]{BF}. Now, following \cite[Dfn. 5.7]{KM}, call $C$ {\it nearly
Gorenstein\/} if $\mu=1$, and following \cite[Dfn. 2.15]{KM}, call $C$ \emph{nearly normal} if
$h^0(\oo/\mathscr{C})=1$. These curves are characterized by important properties, namely, $C$ is
nearly Gorenstein iff it is non-Gorenstein and $C'$ is projectively normal, owing to \cite[Thm. 6.5]{KM},
and $C$ is nearly normal iff $C'$ is arithmetically normal, owing to \cite[Thm. 5.10]{KM}.}
\end{defi}

Let $P\in C$. 
We say that $P$ is
\emph{monomial} provided that the completion
$\widehat{\op} =k[[t_1^{n_{11}}\cdots t_s^{n_{s1}},\,\ldots\,,t_1^{n_{1r}}\cdots t_s^{n_{sr}}]]$,
where $t_{1},\ldots,t_{s}$ are local parameters at $\pb_1,\ldots,\pb_s$.

We also recall the concepts of Clifford index, and Koszul cohomology, applied here for curves within a little bit broadest sense.
Let $\fff$ be a torsion free sheaf of rank 1 on $C$. According to \cite[p. 363 Dfn. 2.2 (7)]{Bal}, we introduce the Clifford Index $C$ as:
$$\cliff(C)=\min\{\deg\,\fff-2(h^0(\fff)-1)\, ;\, h^0(\fff)\geq 2 \rm{\ and\ } h^1(\fff)\geq 2\}$$




\noindent According to \cite{ApF, Gr}, consider the complex
$$
\wedge^{p+1}H^0(\fff)\otimes H^0(\fff^{q-1})\stackrel{\phi_{p,q}^1}{\longrightarrow}\wedge^{p}H^0(\fff)\otimes H^0(\fff^{q})\stackrel{\phi_{p,q}^2}{\longrightarrow} \wedge^{p-1}H^0(\fff)\otimes H^0(\fff^{q+1})
$$
The quotient
$$
K_{p,q}(C,\fff):= \ker(\phi_{p,q}^2)/{\rm im}(\phi_{p,q}^1)
$$
is said to be the \emph{$(p,q)$-th Koszul cohomology} of $\fff$.

We recall Green's conjecture for smooth curves \cite{Gr}:
$$K_{p,2}(C,\omega)=0 \Leftrightarrow p < \cliff(C).$$
One can find a deep study of the whole problem, for instance, in \cite{ApF}. The conjecture was proved for
general regular curves by C. Voisin in \cite{Vo1, Vo2} and even for a class of singular ones as can be seen,
for example, in the recent article \cite{FT} and references therein. 

We will see later on, when dealing with Green's conjecture, that we need to allow torsion free
sheaves of rank $1$ on this definition since $\omega$ is not a bundle if $C$ is non-Gorenstein

\section{Proof of Theorem 1}\label{proofthm1}
%
%

In order to establish Max Noether's Theorem for integral curves,
we first recall some required subject on valuations.

Let $P\in C$ be a point and  $\pb_1,\dots,\pb_s$ the points on $\cb$ over $P$.
Given any non identically null meromorphic function $x\in k(C)$ the \emph{order of
$x$ at $P$} is the $s$-tuple of integers
$v_{P}(x):=(v_{\pb _{1}}(x),\ldots ,v_{\pb_{s}}(x))\in{\mathbb{Z}} ^{s}$, where
$v_{\pb_i}$ is the valuation of the discrete valuation ring $\oo_{\cb,\pb_i}$.
The \emph{semigroup of values} of $P$ is
${\rm S}:=v_{P}(\op )$.
Since $\op$ is a ring, $\rm S$ is a sub-semigroup of $\mathbb{Z}^s$, i.e.
it is closed under addition and the zero-element $(0,\dots,0)$ belongs to $\rm S$.
Additionally, one can verify the following:
\begin{itemize}
\item if $a, b \in \rm S$ then $ \min (a,b) \in \rm S$
\item if $a, b \in \rm S$ and $a_i=b_i$ then there exists $\varepsilon \in \rm S$ such that
$\varepsilon_i > a_i=b_i$ and $\varepsilon_j \geq { min}(a_j,b_j)$ where the equality
happens if $a_j\neq b_j$.
\end{itemize}
We also pick up the following elements of $\rm S$
$$\alpha:={\min}({\rm S}\setminus\{ 0\})\ \ \ \ {\rm and} \ \ \ \ \beta:={\min} (v(\cp)).$$
The partial order we consider here is the natural one: $a\leq b$ if and only if $a_i\leq b_i\, \forall i$.
Note that the elements $\alpha$ and $\beta$ are well defined.

Now given any $a:=(a_1,\ldots,a_s)\in{\mathbb{Z}} ^{s}$ we denote
$$|a|:=a_1+\ldots+a_s$$
and if $\rm E$ is a subset of ${\mathbb{Z}} ^{n}$ one defines
$$
\Delta ^{\rm E}(a):=\{ b\in {\rm E}\ |\  b_{i}=a_{i}\ \rm {for\ some}\ i,\ \rm {and}\ b_{j}>a_{j}\ \rm {if}\ j\neq i\}\, ,
$$
$${\rm E}^*:=\{a\in {\rm E}\,|\, a\leq\beta\} \ \ \ \ {\rm {and}} \ \ \ \ {\rm E}^{\circ}:=\{a\in {\rm E}\, |\, a<\beta\}$$
The \emph{Frobenius vector} of $\rm S$ is $\gamma :=\beta -(1,\ldots ,1)$ and one sets
$${\rm K}= { \rm K}_{P}:=\{ a\in{\mathbb{Z}} ^{s}\ |\ \Delta ^{\rm S}(\gamma -a)=\emptyset\} $$
In order to prove Theorem 1, we start with the following result. 

\begin{lem}\label{maroto}
Let $P\in C$ be a s-branch non-Gorenstein point with semigroup of values $\rm S$.
There exists $d\in \rm K^{\circ}\setminus \rm S$
such that $\beta-d-e_{\ell}\in \rm K^{\circ}$ for some $\ell=1,\dots,s$, where
$\{e_1,\ldots,e_s\}$ is the canonical basis of $\mathbb{N}^s$.
\end{lem}

\begin{proof}
First of all note that if $P$ is unibranch, the existence of such an element $d$ is equivalent to $P$ be a
non-Gorenstein point. Now suppose $P$ is $s$-branch. Take $d \in \rm K^{\circ}\setminus \rm S$ minimal,
i.e, such that there is no element in $\rm K^{\circ}\setminus \rm S$ smaller than $d$.
It is enough to show that, for some $\ell=1, \dots, s$, we have
$$\Delta^{\rm S}(\gamma - (\beta-d-e_{\ell}))= \Delta^{\rm S}(d-(1, \dots,0, \dots, 1))=\emptyset,$$
where $0$ is at the $\ell$-th coordinate.
Suppose by contradiction that for all $\ell=1, \dots, s$ there are elements
$b^{\ell} \in \Delta^{\rm S}(d-(1, \dots,0, \dots, 1))$ in $\rm S$. So each $b^{\ell}$ may be only of two kinds:

1) $b^{\ell}$ is such that $b^{\ell}_{\ell}=d_{\ell}$ and $b^{\ell}_j\geq d_j$ for $i\neq \ell$, i.e,
$$b^{\ell}=(d_1+x_1,\dots, d_{\ell},\dots, d_s+x_s),$$
with $x_i\geq 0$ not simultaneously zero since $d\not\in \rm S$.

2) $b^{\ell}$ is such that $b^{\ell}_i=d_i-1$, with $i \neq \ell$, $b^{\ell}_j\geq d_j$ for $j \neq \ell, i$ and $b^{\ell}_{\ell}>d_{\ell}$, i.e,
$$b^{\ell}=(d_1+x_1,\dots, d_i-1, \dots,d_s+x_s),$$
with $x_{\ell}>0$ and $x_j\geq 0$ for $j \neq \ell,i$. At this case, let
$$s_{\ell}={\rm min}(b^{\ell},d)=(d_1, \dots, d_i-1, \dots,d_s)\in \rm K^{\circ}.$$
By the minimality of $d$, we have $s_{\ell} \in \rm K^{\circ} \cap \rm S$. As $b^{\ell}$ and $s_{\ell}$ have the
same $i$-th coordinates, there is an element $s_{\ell}'=(d_1+y_1,\dots, d_{\ell},\dots,d_s+y_s)$ in
$\rm S$, with $y_i \geq 0$ not simultaneously zero since $d \not\in \rm S$.

Define
$$
c_{\ell}=\left\lbrace\begin{array}{ll}
b^{\ell}, &  \rm {if\ b^{\ell}\ is\ of\ the\ kind\ 1}\\
s_{\ell}', & \rm {if\ b^{\ell}\ is\ of\ the\ kind\ 2}
\end{array}\right.
$$
Therefore, $d={\rm min}(c_1, \dots, c_s) \in \rm S$ and we have a contradiction.
\end{proof}

Now we are able to prove our main result.

\medskip


\noindent {\it Proof of Theorem 1\ \ }
First off, if $C$ is smooth, the statement holds as \cite[p. 117]{ACGH}. Such result is
a consequence of the projective normality of the extremal curves. Rosenlicht proved
at \cite{R} that if $C$ is (nonhyperelliptic) Gorenstein, then  $C'$ is extremal and
$C\cong C'$, so Max Noether's statement holds.
On the other hand, if $C$ is non-Gorenstein one should adjust the proof of \cite[Thm. 3.7]{Mt},
where the statement was proved to unibranch non-Gorenstein points. It's a long
(but not actually hard) task verifying that ``unibranch" is only really needed precisely in the
proof of \cite[Lem. 3.2, stp. 2]{Mt}. So one just have to check that $\cp /t^{\beta-\alpha}\cp$
is generated by elements in $H^0(\vl)^2$. Now from  \cite[Thm. 2.11]{St} or \cite[Pr.p 2.14.(iv)]{BDF}
we have $v_P(\vlp)=\rm K$. So we just have to prove that there exists a sequence
\begin{equation}
\label{equseq}
a_1=\beta<a_2<a_3<\ldots< a_{|\beta|-|\alpha|}<2\beta-\alpha
\end{equation}
such that all the $a_i$ are in $G:=\{a+b\,|\,a,b\in \rm K^{\circ}\}$.

We may assume $\alpha <\beta$ for otherwise the sequence is empty. Call
$$
\alpha^{n}:=\rm {min}(n\alpha, \beta)
$$
with $n\in \mathbb{N}$. Let $r$ be the smallest integer such that $(r+2)\alpha>\beta$ and
write $\beta =\alpha^{r+1} +u$, with $ 0\leq u<\alpha$, where $0=(0,\dots,0)\in \mathbb{N}^s$.
So we may build the sequence (\ref{equseq}) writing each of its element as
$$
a=\beta+\alpha^{n+1}-\alpha+v
$$
for

(i) $0\leq n\leq r-1$ and $0\leq v\leq \alpha-e_{\ell}$

(ii) $n=r$ and $0\leq v<u$,

\noindent where we disregard (i) if $r=0$ and (ii) if $u=0$. Now write
$$
a=\beta+\alpha^{n+1}-\alpha+v=(\beta-\alpha+v)+\alpha^{n+1}.
$$
We have, for every $0\leq n\leq r$, that $\alpha^{n+1}$ is in $\rm S^{\circ}$  and hence in $\rm K^{\circ}$
as well. Moreover, if $v\neq \alpha-e_i$ for every $i$, then $\beta-\alpha+v$ is also in $\rm K^{\circ}$.
Indeed, if there exists $b\in\Delta^{\rm S}(\gamma-(\beta-\alpha+v))=\Delta^{\rm S}(\alpha-v-(1,\dots,1))$
then $b=0$, $v_i=\alpha_i-1$ for some $i$ and  $v_j>\alpha_j-1$ for $j\neq i$. But this happens
if and only if $v= \alpha-e_i$. Note that in (ii), we have $v<u \leq \alpha-e_{i}$ for every $i$, hence
the statement is proved if $r=0$. Hence, it suffices to prove that $\beta+\alpha^{n+1}-e_{\ell}\in G$
for every $0\leq n\leq r-1$. For simplicity, replace $n$ by $n+1$, so it is enough checking that
$\beta+\alpha^n-e_{\ell}\in G$ for every $1\leq n\leq r$.

Let $m$ be the largest integer such that $\alpha^{m+1}=(m+1)\alpha$, i.e, such that $(m+1)\alpha\leq\beta$.
If $n$ is such that $m< n \leq r$, then $\alpha^{n+1}<(n+1)\alpha$ and we know that
$0<\alpha^{n+1}-\alpha^n<\alpha$. So, there exists $v$ with $0\leq v <\alpha-e_{\ell}$ for which
$a=\alpha^n+\beta-e_{\ell}=\alpha^{n+1}+\beta-\alpha+v$. In this case, since $v\neq \alpha-e_{\ell}$,
we have seen that $a\in G$.

On the other hand, let $n$ be such that $1\leq n\leq m\leq r$. Let $d:=d_1$ be as at the Lemma
\ref{maroto} and $d_2:=\beta-d-e_{\ell}$. For every $1\leq n\leq m$ we can find natural numbers
$q_{n1},q_{n2}$ such that $n={q_{n1}}+{q_{n2}}$ and ${q_{nj}}\alpha +d_j<\beta$ for $j=1,2$. In fact,
let $q_{m2}$ be the largest integer such that ${q_{m2}}\alpha\leq d_1$ and $q_{m1}:=m-q_{m2}$.
If $n<m$ we may take $q_{n1}:=\rm {min}\{ n,q_{m1}\}$ and $q_{n2}:=n-q_{n1}\leq q_{m2}$. Thus,
it suffices proving for $n=m$. Assume, without loss in generality, that $d_1\leq d_2$.
As $d_1+d_2=\beta -e_{\ell}$, we have $2d_1<\beta$ and, since $m\geq 1$, we have $2\alpha\leq\beta$.
Moreover, $(q_{m2} +1)\alpha\leq d_1+\alpha\leq 2\,\rm {max}\{ d_1,\alpha\} \leq\beta$ and therefore
$m\geq q_{m2}$ by the very definition of $m$. This implies that $q_{m1}\geq 0$. So, we conclude that
$$
{q_{m2}}\alpha +d_2\leq d_1+d_2<\beta
$$
and
$$\begin{array}{ll}
q_{m1}\alpha +d_1 &=(m-q_{m2})\alpha +d_1 \\
                                   &=m\alpha -{q_{m2}}\alpha +d_1 \\
                                   &\leq(\beta -\alpha )+ (d_1-{q_{m2}} \alpha) < \beta
\end{array}$$
since $0\leq d_1-{q_{m2}}\alpha<\alpha$ by the definition of $q_{m2}$.

Since $\vlp$ is an $\op$-module, the definition of $q_{nj}$ implies $a_{nj}:={q_{nj}}\alpha+d_j\in K^{\circ}$
for every $1\leq n\leq m$ and $j=1,2$. Therefore
$$\begin{array}{ll}
a_{n1}+a_{n2} &={q_{n1}}\alpha+d_1+{q_{n2}}\alpha+d_2\\
              &=({q_{n1}+q_{n2}})\alpha+(d_1+d_2) \\
              &= n\alpha +\beta-e_{\ell} =\beta+\alpha^n-e_{\ell}
\end{array}$$
and the latter is hence  in $G$ whatever are $1\leq n\leq m$, as we wish. \qed

\

In the sequel, we offer an example verifying Max Noether's statement at level $2$ for a simple rational
non-Gorenstein curve. Note that the argument is similar to the (unibranch) case of the proof.

\begin{exam}\emph{
Let $C$ be the projective closure of the affine monomial curve
$$
{\rm {Spec}}\,k[t^3,t^7,t^{10},t^{11}]
$$
It has just one singular point, say $P$, with
$$
\oo_{P}=k\oplus kt^3\oplus kt^6\oplus kt^7\oplus kt^9\overline{\oo}_{P}
$$
and genus $5$.  We have that $P$ is non-Gorenstein since
$t^4 \in \vlp \setminus \oo_P$, and one can check that $H^0(\vl) =\langle 1; t^3; t^4; t^6; t^7\rangle$. 
The rings do not
coincide because $P$ is non-Gorenstein for $\vlp$ is not a free $\op$-module. To prove Max Noether's statement
at level $2$, note that $t^8,t^9,t^{10}, \ldots,t^{14}\in H^0(\vl^2)$. Besides,
$$
\deg_Q(\vl^2)=
\left\lbrace\begin{array}{ll}
2 & \rm {if}\  Q=P \\
14 & \rm {if}\  Q=\infty \\
0 & \rm {otherwise}
\end{array}\right.
$$
therefore $\deg\vl^2=16$. Since $h^0(\vl^2)=16+1-5=12$, the above elements are the ones needed to complete
$H^0(\vl^2)$. We just warn the reader that we had to compute $\deg\vl^2$ because it is not true in general
that it is $2\deg\vl$ if $C$ is non-Gorenstein. For instance, if $C$ is nearly Gorenstein but not Kunz then this
property does not hold. Writing the new elements as below we establish  Max Noether's statement for $n=2$:
$$
\begin{array}{llllllll}
t^{8}=t^4t^4, & & t^{9}=t^3t^6, & & t^{10}=t^3t^7, &  t^{11}=(t^3t^4)(t^4),  \\
t^{12}=(t^3t^3)t^6, & & t^{13}=(t^3t^3)t^7, & & t^{14}=(t^{3}t^4)(t^{3}t^4). &
\end{array}
$$ }
\end{exam}

\section{Applications}\label{app}

This section is devoted to the study of some applications of our main result Theorem 1.
We start with the equivalences below.



\begin{thm}\label{ap1equiv}
The following are equivalent
\begin{itemize}
\item[\rm (i)] $C$ is nonhyperelliptic;
\item[\rm (ii)] There exists a birational morphism $C'\to C$, i.e., Rosenlicht's statement holds;
\item[\rm (iii)] The maps ${\rm Sym}^nH^0(\ww)\longrightarrow H^0(\ww^n)$ are surjective , i.e., Max Noether's statement holds;
\item[\rm (iv)] $K_{0,2}(C,\ww)=0$;
\item[\rm (v)] $\cliff(C)>0$ or $C$ is rational nearly normal.
\end{itemize}
\end{thm}

\begin{proof}
We know that (i) and (ii) are equivalent by Rosenlicht \cite{R}. We also know that (iii) implies (i), and, by what
was mentioned above, implies (ii) as well. By Theorem 1, (i) implies (iii) and the first three items are linked.
Now (iv) is Max Noether's statement at level $2$. This suffices to establish (i)$\Leftrightarrow$(iv) in the smooth
case, since hyperelliptic curves clearly do not satisfy (iv). In a broader context, one needs, for instance,
 to follow the whole intrinsic arguments in \cite[Thm. 3.7]{Mt} to see, as mentioned in the very proof of
 Theorem 1, that Max Noether's statement at level $2$ is sufficient to establish the whole result. So the first four
 items are linked as well. Now, according to \cite[App]{EKS}, $\cliff(C)=0$ if and only if $C$ is hyperelliptic
 or rational nearly normal. Since nearly normal curves are non-Gorenstein, they are nonhyperelliptic as well,
 and the result follows.
\end{proof}

As mentioned in the Introduction, one of the first consequences of the re\-gu\-lar version of
Max Noether's Theorem, which is also valid for Gorenstein curves,
is that a canonical curve $C$ lies in the intersection of some quadrics,
more precisely:

\medskip
Let $I_r(C)$ be the vector space of $r$-forms vanishing on a smooth nonhyperelliptic canonical curve $C$. We have
$$\dim(I_2(C))=\frac{(g-2)(g-3)}{2}.$$


\noindent In order to generalize this result to non-Gorenstein curves, we use the extrinsic part of the proof of
Max Noether's Theorem for nearly Gorenstein curves presented in \cite[Thm. 2.6]{Mt}. Let us fix
some required notation.

Let $\widetilde{C}$ be the curve obtained by the desingularization of all non-Gorenstein points of $C$ through successive
blowups. Thus we obtain a sequence of birational morphisms $\overline{C}\rightarrow\widetilde{C}\rightarrow C'\rightarrow C$.
As usual, if $\widetilde{\pi}:\widetilde{C}\rightarrow C$ is the natural birational morphism, then we set
$\widetilde{\oo}:=\widetilde{\pi}_{*}(\oo_{\widetilde{C}})$ and $\widetilde{\mathscr{C}}:={\rm Hom}(\widetilde{\oo},\oo)$.

Now, our Theorem 2 corresponds to the second item below:

\begin{thm}\label{ap2quadrics}
Let $C$ be a non-Gorenstein curve of genus $g$.
\begin{itemize}
\item[(i)] If $\widehat{C}$ is the blowup of $C$ along $\ww$, then there is an embedding $\widehat{C}\hookrightarrow \mathbb{P}^{g-2+\mu}$ such that
$$\dim(I_r(\widehat{C}))= \left(\begin{array}{c} r+g-2+\mu \\
                                                    r \\ \end{array} \right)-r(2g-2-\eta)+(g-\eta-\mu-1).$$
In particular,
$$\dim(I_2(\widehat{C}))=\frac{g^2+(2\mu-7)g+\mu^2-3\mu+2\eta+6}{2}.$$
\item[(ii)] If the non-Gorenstein points of $C$ are unibranch, then there is an embedding $C\hookrightarrow \mathbb{P}^{g+2(\rho-\sigma)-1}$ such that
$$\dim(I_r(C))=\left(\begin{array}{c} r+g+2(\rho-\sigma)-1 \\
                                                    r \\ \end{array} \right)
  +g(1-2r)-2r(\rho-\sigma)+r-1 $$
  where $\rho=h^0(\ob/\ccc)-h^0(\widetilde{\oo}/\widetilde{\ccc})$ and $\sigma=h^0(\ob/\oo)-h^0(\ob/\widetilde{\oo})$.

In particular,
$$\dim(I_2(C))= \displaystyle\frac{(g+2(\rho-\sigma)-1)(g+2(\rho-\sigma)-2)-2g}{2}. $$
\end{itemize}
\end{thm}

\begin{proof}
\noindent (i) From \cite[Prop. 4.5]{KM}, we have that $\oo_{\widehat{C}}\ww$ is an invertible sheaf on $\widehat{C}$ spanned by $H^0(\ww)$.
Consider the complete linear system $|\oo_{\widehat{C}}\ww|$, which is base point free since $H^0(\ww) \subset H^0(\oo_{\widehat{C}}\ww)$.
 It defines an embedding of $\widehat{C}$ at the space $\mathbb{P}^n$, where $n=h^0(\oo_{\widehat{C}}\ww)-1$.

If $C$ is a non-Gorenstein curve, then by the proof of \cite[Thm. 2.6]{Mt}, $\widehat{C}$ is projectively normal. Thus
$$\dim(I_r(\widehat{C}))= \left(\begin{array}{c} r+n \\
                                                    r \\ \end{array} \right)
 -h^0(({\oo_{\widehat{C}}\ww})^r).$$
At \cite{KM}, is proved that $h^1(({\oo_{\widehat{C}}\ww})^r)=0$ for all $r\geq1$. So, by the Rosenlicht`s Theorem, we have
$$\begin{array}{ll}
  h^0(({\oo_{\widehat{C}}\ww})^r) & =\deg(({\oo_{\widehat{C}}\ww})^r)+1-\widehat{g}+h^1(({\oo_{\widehat{C}}\ww})^r) \\
             & =r(2g-2-\eta)+1-(g-\eta-\mu).
\end{array}$$
Then, $n=g+\mu-2$ and
$$\dim(I_r(\widehat{C}))= \left(\begin{array}{c} r+g+\mu-2 \\
                                                    r \\ \end{array} \right)
 -r(2g-2-\eta)+(g-\eta-\mu-1).$$
In particular, for $r=2$,
$$\begin{array}{ll}
\dim(I_2(\widehat{C})) & = \left(\begin{array}{c} g+\mu \\
                                                    2 \\ \end{array} \right) +(\eta-3g-\mu+3) \\
                        &=\displaystyle\frac{(g+\mu)(g+\mu-1)}{2}+(\eta-3g-\mu+3) \\
                        & =\displaystyle\frac{g^2+(2\mu-7)g+\mu^2-3\mu+2\eta+6}{2}.
\end{array}$$

\noindent (ii) Let $C$ be a curve with a non-Gorenstein point $P$, unibranch, with semigroup of values $S$,
whose gaps are $\mathbb{N}\setminus S=\{l_1, \dots, l_{\delta_P}\}$. Consider the curve $C_*$ with semigroup
of values at $P_* \in C_*$ given by
$$S_*=\{0\}\cup\{2\,\vert\beta\vert-l_i\,\vert\,i=1, \dots, \delta_P \} \cup \{n\in\mathbb{N}\, \vert\, n \geq 2\,\vert\beta\vert+1 \}.$$
Note that $S_*$ is in fact a semigroup of values since $(2|\beta|-l_i)+(2|\beta|-l_j)=4|\beta|-(l_i+l_j)$.
But $l_i+l_j<2|\beta|$, so $(2|\beta|-l_i)+(2|\beta|-l_j)\geq 2|\beta|+1$ and thus it is in $S_*$.
By construction, $C_*$ is such that $\widehat{C_*}=C$, i.e, $C$ is the blowup of $C_*$ along $\ww_*$.
Furthermore, if $P$ is the unique singular point of $C$, then
$g_* = \overline{g}+ 2|\beta|-\delta_P= g+2(|\beta|-\delta_P)$, \ $\eta_*=2(|\beta|-\delta_P)-1$ and $\mu_*=1$.
Note that this is a local argument and it is also true if $C$ has more than one singular point, provided that all
of them are unibranch and non-Gorenstein. But, considering also Gorenstein points, we have
$g_* = \widetilde{g}+ 2\rho-\sigma= g+2(\rho-\sigma)$, \ $\eta_*=2(\rho-\sigma)-1$ and $\mu_*=1$.
Therefore, by (i), it follows that
$$\begin{array}{ll}
\dim(I_r(C)) & =\left(\begin{array}{c} r+g_*+\mu_*-2 \\
                                                    r \\ \end{array} \right)
  -r(2g_*-2-\eta_*)+(g_*-\eta_*-\mu_*-1) \\
                        & = \left(\begin{array}{c} r+g+2(\rho-\sigma)-1 \\
                                                    r \\ \end{array} \right)
  +g(1-2r)-2r(\rho-\sigma)+r-1.
\end{array}$$
At the particular case where $r=2$, we have
$$\begin{array}{ll}
\dim(I_2(C)) & = \left(\begin{array}{c} g+2(\rho-\sigma)+1 \\
                                                    2 \\ \end{array} \right)
  -3g-4(\rho-\sigma)+1 \\
                        & =\displaystyle\frac{(g+2(\rho-\sigma)-1)(g+2(\rho-\sigma)-2)-2g}{2},
\end{array}$$
and we are done.
\end{proof}

\section{On Gonality, Cliford Index and Green's Conjecture}\label{gongliff}

This section is devoted to the study of gonality, Clifford index and Koszul cohomology, in particular
Green's conjecture on canonical curves, by allowing torsion free sheaves in their definitions.


\begin{prop}\label{ap3gonacliff}
Let $C$ be a non-Gorenstein curve of genus $g$.
\begin{enumerate}
\item[\rm (i)] If ${\gon}(C)<g$ then ${\rm Cliff}(C)\leq{\gon}(C)-2$
\item[\rm (ii)]  ${\rm Cliff}(C)=0$ if and only if ${\gon}(C)=2$
\item[\rm (iii)] If $C$ is trigonal and $g\geq 4$, then ${\rm Cliff}(C)=1$
\end{enumerate}
\end{prop}
\begin{proof}
Let $\fff$ be a sheaf on $C$ which computes its gonality. Then we have that
$\deg(\fff)=\gon(C)$ and $h^0(\fff)\geq 2$. By Riemann-Roch,
$$
h^1(\fff)=h^0(\fff)+(g-\gon(C))-1
$$
so $h^1(\fff)\geq 2$ if $\gon(C)<g$. Therefore $\fff$ contributes to the Clifford index, and hence
$$\begin{array}{ll}
\cliff(C) &\leq \deg(\fff)-2h^0(\fff)+2 \\
             &=\gon(C)-2h^0(\fff)+2 \\
             &\leq\gon(C)-2
\end{array}$$
and the item $\rm (i)$ follows.

To prove (ii), as we have already noted, $\cliff(C)=0$ if and only if $C$ is hyperelliptic or rational nearly normal.
But according to \cite[Thm. 3.4]{KM} or \cite[Thm. 2.1]{M}, these are precisely the curves with gonality $2$.
Item (iii) follows directly from (i) and (ii).
 \end{proof}

Now we prove our last result.

\


\noindent {\it Proof of Theorem 3\ \ }
To prove (i). Let $\overline{\fff}:=\oo_{\cb}\langle 1,x\rangle$ be a sheaf which computes gonality in $\cb$.
We may suppose it is supported outside any regular point over a singular point of $C$.  Set also
$$\fff:=\oo_C\langle 1,x\rangle.$$
For any singular point $P\in C$, write
$$\op=k\oplus ky_1 \oplus \cdots \oplus ky_{n} \oplus \cp$$
with $y_i\in\mmp$ for all $i=1, \cdots, n$. Hence
$$\op+x\op=kx+kxy_1+\cdots+kxy_{n}+\op$$
and so
$$\deg_P\fff=\dim((\op+x\op)/\op)\leq \dim(\op/\cp).$$
Thus
$$\begin{array}{ll}
{\gon}(C) &\leq \deg\fff \\
                     & = \deg\overline{\fff}+\sum_{P\in C_{\rm sing}} \deg \fff_P \\
                     & = {\gon}(\cb)+h^0(\oo/\ccc) \\
                     & \leq [ (\overline{g}+3)/2]+g-\overline{g}-\eta \\
                     & = g+1-[\overline{g}/2]-\eta
\end{array}$$
so the result follows and we have equality only when $\eta=1$ and $\overline{g}=0$ or $1$.

To prove (ii), let $m:=\dim(\obp/\mmp\obp)$ be the multiplicity of $P$, where $3\leq m \leq g$. As $C$ is rational, write $k(C)=k(x)$
where $x$ is now the identity function at finite distance of ${\mathbb{P}^{1}}=k\cup\{\infty\}$. Assume the singular point
$P$ lies under $0$.
Then $x^mu\in\op$ for some unit $u$ in $\obp$ since $m$ is the multiplicity of $P$. Now we know that $u$
admits a $m$-th root $u'$ in the completion of $\obp$ by the same argument of a Puiseux parametrization.
But since $u$ is rational, so is $u'$. Replacing $x$ by $xu'$ as the identity function at finite distance, we
may assume $x^m\in\op$. Then $\oo_C\langle 1, x^m \rangle$ has degree $m$ at the point under infinity
and zero at other points of $C$. So $\rm gon(C) \leq m$.
On the other hand, let $n$ be the number of elements in $\mathbb{N}$ between $m$ and $\beta$ outside the
semigroup of values $S$, i.e, the number of gaps of $S$ between $m$ and $\beta$. So, the multiplicity
of $P$ is $m=g+1-n$ and, furthermore, the sheaf ${\oo _{C}\langle 1,x\rangle}$ has degree 1  at the point under infinity, at
most $n+1$ at $P$ and zero at other points of $C$. Therefore,
$$\gon(C) \leq n+2.$$
As $m$ and $n+2$ are inversely proportional, the gonality of $C$ increases as $m$ approaches $n+2$.
The maximum occurs when
$$g+1-n=n+2 \Longleftrightarrow n=\frac{g-1}{2}$$
and $m=g+1-n=\frac{g+3}{2}$. As the gonality is an integer number, it follows that
$$\gon(C) \leq \big\lfloor \frac{g+3}{2} \big\rfloor$$
and we are done.

Let us prove (iii). The reciprocal follows directly from the definition of Clifford index and  Proposition \ref{ap3gonacliff}.
Now, let $\fff$ be a sheaf which computes the Clifford index, then
\begin{equation}
\label{equcl1}
\deg(\fff)=1+2h^0(\fff)-2
\end{equation}
On the other hand, by Clifford's Theorem \cite[App.]{EKS}, since $\cliff(C)\neq 0$, we have
\begin{equation}
\label{equcl2}
h^0(\fff)+h^1(\fff)\leq 5
\end{equation}
If $h^0(\fff)=2$, then $C$ is trigonal by (\ref{equcl1}). If not, $h^0(\fff)=3$ by (\ref{equcl2}), and hence
$\deg(\fff)=5$ by (\ref{equcl1}). In order to see that the latter condition is needed, consider the projectively closure of
$$
{\rm Spec}\, k[t(t-1)^5,t^2(t-1)^3,t^2(t-1)^6,t^2(t-1)^7]
$$
It has genus $5$ and is not trigonal by \cite{AM}, and the sheaf $$\fff:=\oo_C\langle 1,t(t-1)^3,t^2(t-1)^3\rangle$$
has degree $0$ elsewhere but the infinity where it has degree $5$. So $\deg(\fff)=5$ and, by construction, $h^0(\fff)=3$.

To prove (iv). Let $C'$ be the canonical model of $C$, since $C$ is nearly Gorenstein we verify
$$\begin{array}{ll}
\deg(C')&= 2g-2-\eta \\
             &=2(g'+\mu+\eta)-2-\eta \\
            &=2(g'+1+\eta)-2-\eta \\
            &=2g'+\eta
\end{array}$$
Moreover, if $C'$ is nearly Gorenstein, it is defined by a complete linear system owing to \cite[Lem. 5.8]{KM}. Then,
by \cite[Thm. 8.8.1]{E},
\begin{equation}
\label{equkp2}
K_{p,2}(C',\oo_{C'}\ww)=0
\end{equation}
if $p<\eta$ and $C'$ is smooth. But if $C$ is nearly Gorenstein then $C'$ is projectively normal. So one is able to adjust
the proof of \cite[Thm. 8.8.1]{E} to remove the hypothesis that $C'$ should be smooth.

Now, since $C$ is nearly Gorenstein, then $H^0(\oo_{C'}\ww)=H^0(\ww)$ because $C'$ is given by a complete linear system.
 Moreover, $\pi'_{*}((\oo_{C'}\ww)^q)=\ww^q$ for any $q\geq 2$ according to the proof of \cite[Thm. 2.6]{Mt}.
 In particular, $H^0((\oo_{C'}\ww)^q)=H^0(\ww^q)$ for any $q\geq 1$ which implies $K_{p,q}(C',\oo_{C'}\ww)=K_{p,q}(C,\ww)$
 for every $p,q$, and the result follows due to (\ref{equkp2}).

In order to build the family, for every $p\geq 1$, consider the curve $C_p$ which is the projective closure of
$$
{\rm Spec\,} k[t^{p+3}, t^{p+5},t^{p+6}, \dots, t^{2p+7}].
$$
Note that $C_P$ is trigonal since $\oo_{C_P}\langle 1,t\rangle$ has degree $1$ at the infinity and $2$ at the
unique singular point of $C_p$. Besides $C$ is nearly Gorenstein with $\eta= p+1$. Since $C_p$ is trigonal
with genus greater than $4$ we have $\cliff(C_p)=1$, and since $C_p$ is nearly Gorenstein with $p<\eta$ we have
 $K_{p,2}(C_p,\ww)=0$. \qed

\

There are many trigonal non-Gorenstein curves of genus $3$ since not all of them are nearly normal,
which can be easily seen from its very definition. On the other hand, it was proved in \cite{Mt2} that any
 non-Gorenstein curve of genus $4$ is at most trigonal so the bound of (i) is not attained. For genus
 $5$, in \cite{AM} one shows that the curve given by the projective closure of
$$
{\rm Spec}\, k [ t\,u^2+t\,u^3, t\,u^4+t^2\,u^5, t^2\,u^2+t^3\,u^7, t^3\,u^2, t^4\,u^2, t\,u^9, t^2\,u^9 ]
$$
where $u:=t-1$ has genus and gonality $5$. The following example shows that the bound of (ii) is sharp in low genus.


\begin{exam} \rm At the proof of \cite[Thm. 2]{AM}, is shown that the curve with genus 5 given by the projective
closure of ${\rm Spec} \ k[x^4, x^5+x^7,x^{10},x^{11}]$ has gonality $4=\big\lfloor \frac{5+3}{2} \big\rfloor$.
At the same proof, one can observe that, for genus 5, the rational curves whose unique singular point is
monomial have gonality at least 3, i.e, the upper limit can only be reached if the singular point is non-monomial.

Now, let $C$ be a rational curve with genus 6, semigroup of values $\rm S^*=\{0,4,7,8,10\}$ and given
by the projective closure of
$${\rm Spec}\, k [x^4,x^7,x^{10},x^{12},x^{13}].$$
We will show that $C$ has gonality $4=\big\lfloor \frac{6+3}{2} \big\rfloor$. Let $x$ be the identity function at
finite distance of $\mathbb{P}^{1}$ and assume the only singular point $P$ of $C$ lies under $0$. Any sheaf that computes
gonality is generated and always can be taken containing the structural sheaf. So, we need to prove that any
sheaf of the form $\mathscr{G}=\oo _{C}\langle 1,f\rangle$ where $f\in k(x)$ at $C$ has degree at least $4$. For this,
write $f=x^rh$, where $h$ is an unit at $\obp$. If $h$ has no poles at $\infty$, then $h$ does not affect the
degree of $\mathscr{G}$, since their poles at finite points of $\mathbb{P}^{1}$ compensate the losses at infinity; and if $h$ has a
pole at infinity, it only can add degree to $\mathscr{G}$. Thus, we can assume $\mathscr{G}=\oo _{C}\langle 1,x^r\rangle$.
Note that, for $r=1$, $\fff={\oo _{C}\langle 1,x\rangle}$ has degree 1 at infinity, zero at the other points but $P$ and 3 at $P$, since
$$0<1<4<5<7<8<9<10$$
is a saturated sequence of elements of $A:=v_P(\fff_P)=v_P(\op+x\op)$ linking the minimum element of $A$
to the conductor $\beta$ of $S$, with $|A\setminus \rm S |=3$, according \cite[Prp. 2.11(iii)]{BDF}, we have
$3=\dim(\fff _{P}/\oo _{P})=\deg _{P}(\fff)$. Therefore, $\deg(\fff)=4$ and $\gon(C)\leq 4$.
On the other hand, if $r \geq 4$, then $\deg_{\infty}(\mathscr{G})\geq 4$; $r=3$ implies that $\deg_P(\mathscr{G})=1$ and
$\deg_{\infty}(\mathscr{G})=3$; if $r=2$, then $\deg_P(\mathscr{G})=3$ and $\deg_{\infty}(\mathscr{G})=2$; finally, if $r\leq -1$, so
$\deg_P(\mathscr{G}) \geq 4$. Thus, $\gon(C)=4$.
\end{exam}







\end{document}